\documentclass{amsart}
\usepackage{amsmath,graphicx,amsfonts,amssymb,amsthm,hyperref,amscd,verbatim,tikz,tikz-cd}
\providecommand{\U}[1]{\protect\rule{.1in}{.1in}}
\def\theenumi{\arabic{enumi}}

\def\theenumii{\alph{enumii}}
\def\p@enumii{\theenumi.}

\def\theenumiii{\arabic{enumiii}}
\def\p@enumiii{(\theenumi)(\theenumii)}

\def\p@enumiv{\p@enumiii.\theenumiii}
\newcommand{\dmo}{\DeclareMathOperator}

\newcommand{\Q}{\mathbb{Q}}

\newcommand{\Ca}{\mathcal{C}}

\newcommand{\as}{\text{*}}

\newcommand{\Sn}{\mathfrak{S}}
\newcommand{\arXiv}[1]{\href{http://arxiv.org/abs/#1}{\nolinkurl{arXiv:#1}}}
\newcommand{\arXivV}[2]{\href{http://arxiv.org/abs/#1}{\nolinkurl{arXiv:#1v#2}}}
\dmo\FI{FI}
\dmo\Mod{-Mod}
\dmo\fMod{-mod}
\dmo\Ind{Ind}
\dmo\Res{Res}
\dmo\Tab{Tab}
\dmo\Hom{Hom}
\dmo\wt{weight}

\dmo\End{End}

\newtheorem{theorem}{Theorem}[section]
\newtheorem*{thm}{Theorem}
\newtheorem{thmab}{Theorem}

\newtheorem{corollary}[theorem]{Corollary}
\newtheorem{lemma}[theorem]{Lemma}
\newtheorem{proposition}[theorem]{Proposition}
\theoremstyle{definition}
\newtheorem{definition}[theorem]{Definition}
\newtheorem{remark}[theorem]{Remark}

\renewenvironment{proof}[1][\proofname]{{\bfseries #1\\}}{\qed}

\title{Generalized Representation Stability and $\FI_d$-modules}
\date{}
\author{Eric Ramos}
\address{Department of Mathematics, University of Wisconsin - Madison.}
\email{eramos@math.wisc.edu}

\thanks{The author was supported by NSF grant DMS-1502553}
\begin{document}

\maketitle

\begin{abstract}
In this note we consider the complex representation theory of $\FI_d$, a natural generalization of the category $\FI$ of finite sets and injections. We prove that finitely generated $\FI_d$-modules exhibit behaviors in the spirit of Church-Farb representation stability theory, generalizing a theorem of Church, Ellenberg, and Farb which connects finite generation of $\FI$-modules to representation stability.
\end{abstract}

\section{Introduction}

Let $d$ be a fixed positive integer. We write $\FI_d$ denote the category whose objects are the sets $[n] = \{1,\ldots, n\}$ and whose maps are pairs $(f,g):[n] \rightarrow [m]$ of an injection $f:[n] \rightarrow [m]$ along with a $d$-coloring of the compliment of the image of $f$. Composition in this category is defined in the most natural way (see Definition \ref{catfid}). If $k$ is a commutative ring, then an \textbf{$\FI_d$-module} over $k$ is a functor $V:\FI \rightarrow \text{Mod}_k$. For any object $[n]$ of $\FI_d$, one notes that $\End_{\FI_d}([n]) = \Sn_n$. It follows from this that each of the $k$-modules $V([n])$ is, in fact, a $k[\Sn_n]$-module.\\

One immediately observes that if $d = 1$, then $\FI_1$ is naturally equivalent to the category $\FI$ of finite sets and injections. $\FI$-modules were first discussed by Church, Ellenberg, and Farb in \cite{CEF} due to their connection with Church-Farb representation stability \cite{CF}. Since the paper of Church, Ellenberg, and Farb, $\FI$-modules have been observed to be applicable in a wide range of subjects (see, for example, \cite{CEFN} \cite{CE} \cite{GL} \cite{W} \cite{N}). The goal of this paper is to consider the relationship between $\FI_d$-modules and a kind of generalized representation stability.\\

We will write write $V_n$ to denote the value of the functor $V$ on the set $[n]$. The \textbf{elements} of the module $V$ are the members of the set $\sqcup_n V_n$. We say that an $\FI_d$-module $V$ is \textbf{finitely generated} if there is a finite set of elements which generate the entire module (see Definition \ref{fg}). The following theorem is one of the main results of \cite{CEF}, relating finite generation of $\FI$-modules to representation stability.\\

\begin{thm}[\cite{CEF}, Theorem 1.13]
Let $V$ be an $\FI$-module over a field $k$ of characteristic 0, and let $\phi_n:V_n \rightarrow V_{n+1}$ denote the map induced by the standard inclusion $[n] \hookrightarrow [n+1]$. Then $V$ is finitely generated if and only if $V_n$ is finite dimensional for all $n$, and for all $n \gg 0$
\begin{enumerate}
\item $\phi_n$ is injective;
\item $\phi_n(V_n)$ spans $V_{n+1}$ as an $\Sn_{n+1}$-representation;
\item there a decomposition:
\[
V_n \cong \bigoplus_{\lambda}c_\lambda S(\lambda)_n
\]
where the sum is over partitions $\lambda \vdash m$ with $m \leq n$, and the coefficients $c_\lambda$ do not depend on $n$.\\
\end{enumerate}
\end{thm}

If $\lambda = (\lambda_1,\ldots,\lambda_h)$ is a partition of size $|\lambda| := \sum_i \lambda_i$, and $n \geq |\lambda| + \lambda_1$ is some positive integer, then $S(\lambda)_n$ is, by definition, the irreducible complex $\Sn_n$-representation associated to the partition
\[
\lambda[n] := (n-|\lambda|,\lambda_1,\ldots,\lambda_h).
\]
One important take away from this theorem is that the growth of the modules $V_n$ eventually becomes predictable, and this prediction is based on the combinatorial data of the partitions $\lambda$.\\

We will find that more general $\FI_d$-modules are far less rigid in what irreducible representations are allowed to appear as summands of their constituent modules. To make this precise, we will first need to generalize the representations $S(\lambda)_n$. If $n_1 \geq \ldots \geq n_r \geq |\lambda| + \lambda_1$ is a sequence of positive integers, then we define $S(\lambda)_{n_1,\ldots,n_r}$ to be the the irreducible representation of $\Sn_{(\sum_i n_i) - (r-1)|\lambda|}$ associated to the partition 
\[
\lambda[n_1,\ldots,n_r] := (n_1 - |\lambda|,\ldots, n_r - |\lambda|,\lambda_1,\ldots,\lambda_h).
\]
The majority of this paper will be working towards the proof of the following theorem.\\

\begin{thmab}\label{genrepstab}
Let $V$ be an $\FI_d$-module over a field $k$ of characteristic 0, and write $\phi^i_n:V_n \rightarrow V_{n+1}$ for the map induced by the pair of the standard inclusion $[n] \hookrightarrow [n+1]$ and the color $i$. Then $V$ is finitely generated if and only if $V_n$ is finite dimensional for all $n \geq 0$, and for all $n \gg 0$:
\begin{enumerate}
\item $\cap_i \ker\phi^i_n = \{0\}$;
\item $\sum_i \phi_n^i(V_n)$ spans $V_{n+1}$ as an $\Sn_{n+1}$-representation;
\item for any partition $\lambda$, and any integers $n_1 \geq \ldots \geq n_d \geq |\lambda| + \lambda_1$, let $c_{\lambda,n_1,\ldots,n_d}$ be the multiplicity of $S(\lambda)_{n_1,\ldots,n_d}$ in $V_{\sum_i n_i - (d-1)|\lambda|}$. Then the quantity $c_{\lambda,n_1+l,\ldots,n_d+l}$ is independent of $l$ for $l \gg 0$.\label{c3}\\
\end{enumerate}
\end{thmab}

Perhaps the most intriguing condition in the above theorem is \ref{c3}. It can be thought of as saying that the irreducible representations which appear as summands of the constituent modules of $V$, corresponding to partitions with at least $d$-rows, eventually appear with predictable multiplicity. If $d > 1$, then the above theorem does not say anything about multiplicities of irreducible representations corresponding to smaller partitions. For example, if $\lambda$ is a partition, and $n_1 \geq \ldots \geq n_r$ are positive integers with $r \leq d$, it is natural to ask how the quantity $c_{\lambda,n_1 + l,\ldots, n_r + l}$ depends on $l$. While we do not answer this question in general, we do find an answer in one notable case.\\

\begin{thmab}\label{polystab}
Let $V$ be a finitely generated $\FI_d$-module, let $\lambda$ be a partition of some integer $m$, and let $c_{\lambda,n}$ denote the multiplicity of $S(\lambda)_n$ in $V_n$. Then there exists a polynomial $p(x) \in \Q[x]$ of degree $\leq d-1$ such that for all $n \gg 0$, $c_{\lambda,n} = p(n)$.\\
\end{thmab}

Note that if $d \geq 2$ then one can construct examples where the polynomial of the above theorem is non-constant. This is already a departure from the case of $\FI$-modules, as one observes from the third condition in the theorem of Church, Ellenberg, and Farb.\\

\section*{Acknowledgments}
The author would like to give thanks to Jordan Ellenberg and Rohit Nagpal for many useful conversations during the writing of this paper. The author would also like to give special thanks to Steven Sam and Andrew Snowden for informing him of their result, Theorem \ref{ktheory}. Very special thanks should also be given to Steven Sam, who first observed most of the ideas used in Section \ref{thmb}. Finally, the author would like to acknowledge the generous support of the National Science Foundation, through NSF grant DMS-1502553.\\

\section{Basic Definitions and Notation}

\subsection{The Representation Theory of the Symmetric Groups}

We begin with a brief review of the complex representation theory of the symmetric groups. For the remainder of this section, we assume that $k$ is a field of characteristic 0. A reference for much of the material which appears in this section is \cite{CST}. It is a well known fact that the complex irreducible representations of $\mathfrak{S}_n$ can be defined over $\Q$. For this reason, all of what follows can be done over the field $k$.\\

\begin{definition}
For any positive integer $n$, a \textbf{partition} of $n$, denoted $\lambda \vdash n$, is a tuple $\lambda = (\lambda_1,\ldots,\lambda_h)$ of positive integers such that $\lambda_i \geq \lambda_{i+1}$ for each $i$ and $\sum_i \lambda_i = n$. Similarly, a \textbf{composition} of $n$ is a tuple $(a_1,\ldots,a_h)$ of non-negative integers such that $\sum_i a_i = n$. Observe that the only differences between these two concepts is that partitions are ordered, and compositions allow for zero entries. If $\lambda = (\lambda_1,\ldots,\lambda_h)$ is a partition, or a composition, of $n$ then we use $l(\lambda) = h$ to denote its \textbf{length}. We will also use $|\lambda| = n$ to denote the \textbf{size} of $\lambda$.\\
\end{definition}

\begin{theorem}
There is a one to one correspondence between partitions of $n$ and irreducible representations of the symmetric group $\Sn_n$.\\
\end{theorem}

Given a partition $\lambda \vdash n$, we will use $S^\lambda$ to denote the irreducible representation of $\Sn_n$ associated to $\lambda$. By convention, the partition $(n)$ will correspond to the trivial representation of $\Sn_n$.\\

The correspondence between irreducible representations of $\Sn_n$ and partitions of $n$ implies many strong connections between the combinatorics of a partition, and the algebra of the associated irreducible representation. Many of these connections are stated in terms of Young tableau.\\

\begin{definition}
Given a partition $\lambda = (\lambda_1,\ldots,\lambda_h) \vdash n$, we visualize $\lambda$ as a left justified diagram comprised of rows of boxes of equal size, such that row $i$ has precisely $\lambda_i$ boxes. Such a diagram is known as the \textbf{Young tableaux} associated to $\lambda$. The box in position $(i,j)$ is defined to be that which is $i-1$-rows down, and $j-1$-columns to the right, of the box in the top left position. If $\lambda$ is a partition whose associated tableaux has a box in position $(i,j)$, then we write $(i,j) \in \lambda$.\\

A \textbf{filling} of the Young tableaux associated to $\lambda$ is a bijection between the boxes of the tableaux and the set $[n] = \{1,\ldots, n\}$. Any filling for which the numbers are increasing down every column and row is called \textbf{standard}.\\
\end{definition}

\begin{theorem}
For a partition $\lambda \vdash n$, write $\text{Tab}(\lambda)$ to denote the set of standard fillings of $\lambda$. Then,
\[
\dim_k S^\lambda = |\Tab(\lambda)| 
\]
\textbf{}\\
\end{theorem}

One approach to computing $|\Tab(\lambda)|$ is through the hook length formula.\\

\begin{definition}
Let $\lambda$ be a partition, and assume that $(i,j) \in \lambda$. Then we define the \textbf{hook} at $(i,j)$ to be the sub-diagram consisting of this box, as well as all boxes $(k,l) \in \lambda$ such that either $i = k$ and $l \geq j$ or $l = j$ and $l \geq i$. We define the \textbf{length} of this hook, denoted $H(i,j)$, to be the total number of boxes it contains.\\
\end{definition}

\begin{theorem}[The Hook Length Formula]
Let $\lambda \vdash n$. Then the dimension of the irreducible representation $S^\lambda$ is given by the formula
\[
\dim_kS^\lambda = \frac{n!}{\prod_{(i,j) \in \lambda} H(i,j)}.
\]
\text{}\\
\end{theorem}

\begin{definition}
Two partitions $\lambda = (\lambda_1,\ldots,\lambda_h) \vdash n$ and $\mu = (\mu_1,\ldots,\mu_l) \vdash m$ are related $\mu \leq \lambda$ if and only if $l \leq h$ and $\mu_i \leq \lambda_i$ for all $i$. This is equivalent to requiring that the tableaux of $\mu$ fit inside the tableaux of $\lambda$.\\
\end{definition}

For the purposes of this work, we will need a version of Pieri's rule which is slightly more general than that which is usually encountered. If $a = (a_1,\ldots,a_h)$ is a composition of $n$ we will use $\Sn_a$ to denote the subgroup $\mathfrak{S}_{a_1} \times \Sn_{a_2} \times \ldots \times \Sn_{a_h} \leq \Sn_n$. The following theorem follows from the usual Pieri's rule, as well as a simple induction argument.\\

\begin{theorem}[Pieri's rule]
Let $\mu = (\mu_1,\ldots,\mu_l) \vdash m$, and let $a = (a_1,\ldots,a_h)$ be a composition of $n-m$. Then
\[
\Ind_{\mathfrak{S}_m \times \mathfrak{S}_{a}}^{\mathfrak{S}_n} S^\mu \boxtimes k = \bigoplus S^\lambda
\]
where the sum is over chains of the form
\[
\mu = \mu^{(0)} \leq \mu^{(1)} \leq \ldots \leq \mu^{(h-1)} \leq \mu^{(h)} = \lambda
\]
such that $\mu^{(i)}$ is obtained from $\mu^{(i-1)}$ by distributing $a_i$ boxes to distinct columns.\\
\end{theorem}

\begin{remark}
Note that it is not always the case that the right hand side of Pieri's rule is multiplicity free. Indeed, the multiplicity of $S^\lambda$ in $\Ind_{\mathfrak{S}_m \times \mathfrak{S}_{a}}^{\mathfrak{S}_n} S^\mu \boxtimes k$ is precisely the number of chains between $\mu$ and $\lambda$ satisfying the conditions stated in the theorem.\\
\end{remark}

\subsection{$\FI_d$-Modules}

Fix a positive integer $d$ for the remainder of this paper. 

\begin{definition}\label{catfid}
The category $\FI_d$ is defined as follows. Objects of $\FI_d$ are finite sets, while morphisms in $\FI_d$ are pairs $(f,g)$ where $f:S \hookrightarrow T$ is an injection, and $g$ is a $d$-coloring of the compliment of the image of $f$, i.e. a map from the compliment of the image of $f$ to the set $[d] := \{1,\ldots, d\}$.\\

If $(f,g)$ and $(f',g')$ are two composable morphisms, then we set
\[
(f,g) \circ (f',g') = (f \circ f',h),
\]
where 
\[
h(x) = \begin{cases} g(x) &\text{ if $x \notin \text{im}f$}\\ g'(f^{-1}(x)) &\text{ otherwise.} \end{cases}
\]
\end{definition}

We observe that $\FI_1$ is equivalent to the category $\FI$ of finite sets and injections. We also observe that $\FI_d$ has a fully faithful subcategory, whose objects are the sets $[n] = \{1,\ldots,n\}$. We also refer to this subcategory as being $\FI_d$.\\

\begin{definition}
Given a commutative ring $k$, an \textbf{$\FI_d$-module over $k$} is a covariant functor $V:\FI_d \rightarrow \text{Mod}_k$. Given an $\FI_d$-module $V$, we will use $V_n$ to denote $V([n])$.\\

If $(f,g):[n] \rightarrow [m]$ is a morphism in $\FI_d$, we write $(f,g)_\as$ to denote the map $V(f,g)$. These maps are known as the \textbf{induced maps} of $V$. In the case where $n < m$, we refer to $(f,g)_\as$ as a \textbf{transition map} of $V$.\\
\end{definition}

Observe that the action of the $\FI_d$-endomorphisms of $[n]$, for any $n$, make $V_n$ into an $\Sn_n$-representation over $k$. In this way, we may think of $\FI_d$-modules as being sequences of $\Sn_n$-representations, with $n$ increasing, which are compatible with one another under the actions of the transition maps.\\

We will write $\FI_d\Mod$ to denote the category of $\FI_d$-modules with natural transformations. Because the objects of $\FI_d\Mod$ are valued in an abelian category, it follows that $\FI_d\Mod$ is itself an abelian category.

\begin{definition}
Let $m$ be a non-negative integer. The \textbf{free $\FI_d$-module generated in degree $m$}, is defined on objects by
\[
M(m)_n := k[\Hom_{\FI_d}([m],[n])],
\]
the free $k$-module with basis indexed by the set $\Hom_{\FI_d}([m],[n])$. For each morphism $(f,g)$, the induced map $(f,g)_\as$ is defined on basis vectors by composition.\\

More generally, let $W$ be a $k[\mathfrak{S}_m]$-module. Then we define the \textbf{free $\FI_d$-module relative to $W$}, $M(W)$, as follows. For each positive integer $n$ we set 
\[
M(W)_n := M(m)_n \otimes_{k[\Sn_m]} W.
\]
The induced maps are one again defined by composition on the first coordinate. Direct sums of free $\FI_d$-modules of either type will also be referred to as being \textbf{free}.\\
\end{definition}

Free modules are vitally important to the theory of $\FI_d$-modules. In fact, their analogs appear to be fundamental objects in the representation theory of many other combinatorial categories. See \cite{R} and \cite{LY} for examples of this. One should note that there is a canonical isomorphism $M(k[\Sn_m]) \cong M(m)$.\\

\begin{definition}\label{fg}
An $\FI_d$-module $V$ is said to be \textbf{generated in degree $\leq m$} if there is a list of integers $\{m_i\}_{i \in I}$, with $m_i \leq m$ for all $i \in I$, and an exact sequence of $\FI_d$-modules
\begin{eqnarray}
0 \rightarrow K \rightarrow \bigoplus_{i \in I}M(m_i)^{n_i} \rightarrow V \rightarrow 0 \label{presentation}
\end{eqnarray}
If the indexing set $I$ can be taken to be finite, then we say that $V$ is \textbf{finitely generated}.\\
\end{definition}

\begin{remark}
Saying $V$ is generated in degree $\leq m$ is equivalent to there existing a set $\{v_i\}_{i \in I} \subseteq \sqcup_{n = 0}^m V_n$ such that no proper submodule of $V$ contains every element of $\{v_i\}_{i \in I}$. This follows from the following important adjunction, for any $k[\Sn_m]$-module $W$,
\[
\Hom_{\FI\Mod}(M(W),V) = \Hom_{\Sn_m}(W,V_m).
\]
In particular, constructing the surjection in (\ref{presentation}) is equivalent to choosing the set $\{v_i\}_{i \in I}$.\\
\end{remark}

We write $\FI_d\fMod$ to denote the category of finitely generated $\FI_d$-modules. One of the main theorems about finitely generated $\FI_d$-modules is the following.\\

\begin{theorem}[\cite{SS}, Theorems 7.1.2 and 7.1.5]\label{noeth}
If $V$ is a finitely generated $\FI_d$-module over a Noetherian ring $k$, then all submodules of $V$ are also finitely generated. Moreover, if $k$ is a field, then there exists polynomials $p_1^V,\ldots,p_d^V \in \Q[x]$ such that
\[
\dim_k(V_n) = p_1^V(n) + p_2^V(n)2^n + \ldots + p_d^V(n)d^n
\]
for all $n\gg 0$.
\end{theorem}

Prior to the provided source, the Noetherian property in the above theorem was proven for $\FI_d$-modules over a field of characteristic 0 in \cite[Theorem 2.3]{S}. It was proven for $\FI$-modules over a field of characteristic 0 in \cite[Theorem 1.3]{CEF}. Later \cite[Theorem A]{CEFN} proved this result for $\FI$-modules over arbitrary Noetherian rings. The second part of the theorem, on dimensional stability, was proven in the case where $k$ is a field of characteristic 0 in \cite[Theorem 3.1]{S}. It was proven for $\FI$-modules over a field of characteristic 0 in \cite[Theorem 1.5]{CEF}, and over an arbitrary field in \cite[Theorem B]{CEFN}.\\

\begin{definition}
We call the function
\[
n \mapsto \dim_k(V_n)
\]
the \textbf{Hilbert function} of $V$. We say the Hilbert function of $V$ is $o(d^n)$ if the polynomial $p_d^V \in \Q[x]$ from Theorem \ref{noeth} is zero.\\
\end{definition}

\section{Representation Stability and $\FI_d$-modules}

In \cite{CF}, Church and Farb describe the phenomenon of representation stability. Following this, Church, Ellenberg, and Farb proved that representation stability could be equivalently stated in terms of finite generation of $\FI$-modules \cite[Theorem 1.13]{CEF}. The goal of the first half of this section of the paper is to prove Theorem \ref{genrepstab}, which suggests a kind of generalized representation stability. Following this we prove Theorem \ref{polystab}.\\

\emph{We assume throughout the remainder of the paper that $k$ is a field of characteristic 0.}\\
t
\subsection{Padded Partitions and Representations}
t
\begin{definition}\label{dpad}
Let $\lambda = (\lambda_1,\lambda_2,\ldots,\lambda_h)$ be any partition, and let $(n_1,\ldots,n_r)$ be an $r$-tuple of integers such that $n_1 \geq n_2 \geq \ldots \geq n_r \geq |\lambda| + \lambda_1$. Then we define the \textbf{$r$-padded partition}
\[
\lambda[n_1,\ldots,n_r] := (n_1-|\lambda|,n_2-|\lambda|,\ldots,n_r-|\lambda|,\lambda_1,\ldots,\lambda_h)
\]
Similarly, we define the \textbf{$r$-padded representation}
\[
S(\lambda)_{n_1,\ldots,n_r} := \begin{cases} S^{\lambda[n_1,\ldots,n_r]} &\text{ if $n_1 \geq \ldots \geq n_r \geq |\lambda|+\lambda_1$}\\ 0 &\text{ otherwise.}\end{cases}
\]
\end{definition} 

In this paper, we think of padding as a kind of parametrization. We have the following easy observation.\\

\begin{lemma}
Let $\mu =(\mu_1,\ldots,\mu_l)$ be a partition with at least $r$-rows. Then there exists a unique partition $\lambda = (\lambda_1,\ldots,\lambda_h)$, as well as a unique collection of integers $n_1 \geq n_2 \geq \ldots \geq n_r \geq |\lambda| + \lambda_1$, such that $\mu = \lambda[n_1,\ldots,n_r]$.\\
\end{lemma}

As a specific instance of the above lemma, if $\mu \vdash n$, then there is a unique partition $\lambda$ such that $\mu = \lambda[n]$. This notation is convenient, as it allows us to uniformly describe the irreducible representations which appear in an $\FI_d$-module. In particular, given an $\FI_d$-module $V$, for all $n \geq 0$ we may write
\[
V_n = V^{<d}_n \oplus \left(\bigoplus_{\lambda, n_1,\ldots,n_d} c_{\lambda,n_1,\ldots,n_d}S(\lambda)_{n_1,\ldots,n_d}\right),
\]
where we implicitly only allow $\lambda,n_1 ,n_2,\ldots,n_d$ such that $|\lambda[n_1,\ldots,n_d]| = n$, and the irreducible constituents of $V^{<d}_n$ are associated to partitions with strictly less than $d$ rows. One may think of the above equation as expressing $V_n$ as the sum of a grouping of all ``small'' length partitions, and a grouping of all ``large'' partitions. Theorem \ref{genrepstab} will tell us that if $V$ is finitely generated, then the multiplicities which appear in the above decomposition are stable in the appropriate sense.\\

Note that in the next section we will prove that if $V$ is finitely generated, then the partitions $\lambda$ which appear in the above decomposition have bounded size. This is analogous to the theorem from \cite[Proposition 3.2.5]{CEF}, which states that the weight of a finitely generated $\FI$-module is finite.\\

\begin{definition}
Let $V$ be an $\FI_d$-module and decompose $V_n$ in the following way,
\[
V_n = V_n^{< d} \oplus \left( \bigoplus_{\lambda,n_1,\ldots,n_d} c_{\lambda,n_1,\ldots,n_d}S(\lambda)_{n_1,\ldots,n_d} \right),
\]
where all the irreducible constituents of $V_n^{\leq d}$ correspond to partitions with strictly less than $d$ rows. Then the \textbf{$d$-weight} $\wt^d(V)$ of $V$ is the maximum value of $|\lambda|$ across all $S(\lambda)_{n_1,\ldots,n_d}$ appearing in $V_n$ for all $n \geq 0$. If no such integer exists, then we say $\wt^d(V) =\infty$.\\
\end{definition}

\subsection{Free Modules in Characteristic 0}

We spend this section proving various lemmas and propositions about the $\FI_d$-module $M(W)$. These will prove important to us in the sections which follow. We begin with a formula for $M(W)$ which is similar to that given in \cite{CEF} for $\FI$-modules.\\

\begin{proposition}\label{indform}
Let $W$ be a $k[\mathfrak{S}_m]$-module. Then for each $n$,
\[
M(W)_n = \bigoplus_{a = (a_1,\ldots,a_d)}\Ind^{\mathfrak{S}_n}_{\mathfrak{S}_m \times \mathfrak{S}_a} W \boxtimes k,
\]
where the sum is over compositions $a$ of $n-m$ of length $d$.
\end{proposition}

\begin{proof}
Looking at the $\Sn_n$-set $\Hom_{\FI_d}([m],[n])$, one quickly observes that the action of $\Sn_n$ does not alter the colors which appear in a given morphism. In fact, it is easily checked that the orbits are in bijection with compositions of $n-m$ of length $d$. These orbits split the $\Sn_n$-representation $M(m)_n$ into disjoint pieces. We will denote the piece associated to the composition $a:=(a_1,\ldots,a_d)$ by $M(m)_{n,a}$.\\

We may write
\[
M(W)_n = M(m)_n \otimes_{k[\mathfrak{S}_m]} W = M(m)_{n,a} \otimes_{k[\mathfrak{S}_m]} W
\]
where the sum is over the same set as in the statement of the proposition. We may consider the effect of tensoring $W$ with $M(m)_{n,a}$ over $k[\mathfrak{S}_m]$ as equating elements $(f,g)$ and $(h,g')$ whenever $f$ and $h$ have the same image and $g = g'$. It follows that a basis for $M(W)_n$ is given by pairs $w_i \otimes (f,g)$, where $\{w_i\} \subseteq W$ is a basis for $W$, and $(f,g)$ is such that $f$ is monotone.\\

Fix a composition $a$ of $n-m$, and consider the term $W \otimes_{k[\mathfrak{S}_m]} M(m)_{n,a}$. Let $(f,g):[m] \rightarrow [n]$ be the pair of the standard inclusion, i.e. the inclusion which sends $j$ to $j$, and the coloring which colors $m+1,\ldots, m+a_1$ the color 1, $m+a_1 + 1,\ldots m+a_1 + a_2$ the color 2, and so on. Then the subgroup $\Sn_m \times \Sn_a$ acts on the pure tensors $w \otimes (f,g)$ in the same way that it acts on $W \boxtimes k$. This gives us a map
\[
\Ind^{\mathfrak{S}_n}_{\mathfrak{S}_m \times \mathfrak{S}_a} W \boxtimes k \rightarrow W \otimes_{k[\mathfrak{S}_m]} M(m)_{n,a}
\]
It is easily seen that this map is surjective. On the other hand, the vector space on the left hand side has dimension
\[
\dim(W)|\mathfrak{S}_n|/|\mathfrak{S}_m \times \mathfrak{S}_a| = \dim(W) \binom{n}{m,a_1,\ldots,a_d}.
\]
This is the same as the dimension of $W \otimes_{k[\mathfrak{S}_m]} M(m)_{n,a}$.\\
\end{proof}

Observe that if $d = 1$, the above formula becomes
\[
M(W)_n = \Ind_{\mathfrak{S}_m \times \mathfrak{S}_{n-m}}^{\mathfrak{S}_n}W \boxtimes k
\]
This is the definition of $M(W)$ given in \cite{CEF} and \cite{CEFN}.\\

The above description of the free module $M(W)$ clearly lends itself to applications of Pieri's rule. Indeed, using this simple combinatorial tool, we will be able to prove many representation theoretic properties of free modules.\\

\begin{proposition}\label{mwt}
Let $\lambda \vdash m$. Then $\wt^d(M(S^\lambda)) = m$.
\end{proposition}

\begin{proof}
For the purposes of this proof, we write $M(\lambda) := M(S^\lambda)$. According to Proposition \ref{indform} we may write
\[
M(\lambda)_n = \bigoplus_a \Ind_{\mathfrak{S}_m \times \mathfrak{S}_a}^{\mathfrak{S}_n} S^\lambda \boxtimes k
\]
Thus, applying Pieri's rule, the irreducible factors $S^\mu$ which appear in $M(\lambda)_n$ will precisely be those for which there is a chain
\[
\lambda \leq \lambda^{(1)} \leq \ldots \leq \lambda^{(d-1)} \leq \mu
\]
such that $\lambda^{(i)}$ is obtained from $\lambda^{(i-1)}$ by adding some number (perhaps zero) of boxes to distinct columns. Assume that $\mu$ is a partition with at least $d$ rows which is associated to an irreducible in $M(\lambda)_n$ and write $\mu = \nu[n_1,\ldots,n_d]$. $\mu$ will be a constituent of $M(\lambda)_n$ if and only if one can obtain $\lambda$ from $\mu$ by removing blocks from distinct columns in at most $d$ steps. One can apply a ``greedy'' block removal algorithm which, in each step, removes all blocks which can be removed. It is clear that $\mu$ is a constituent of $M(\lambda)_n$ if and only if such an algorithm eventually yields $\lambda$. Observe that after the first step of the greedy algorithm applied to $\mu$, one is left with the partition $(\mu_2,\ldots,\mu_h)$. Therefore, after $d$ steps one would be left with the partition $\nu$. It follows that $|\nu| \leq |\lambda|$, and so $\wt^d(M(\lambda)) \leq |\lambda|$. Conversely, it is clear that the greedy algorithm applied to any $d$-padding of $\lambda$ itself eventually terminates at $\lambda$ precisely. This shows that $\wt^d(M(\lambda)) = |\lambda|$.\\
\end{proof}

Observe that the above proposition implies that all finitely generated $\FI_d$-modules have bounded $d$-weight.\\

\begin{corollary}
If $V$ is generated in degree $\leq m$, then $\wt^d(V) \leq m$.\\
\end{corollary}

The proof of the above proposition also motivates the following observation about the irreducible constituents of free modules.\\

\begin{proposition}\label{coeffstab}
Let $W$ be a finite dimensional $k[\Sn_m]$-module, let $\mu$ be a partition, and let $c_{\mu,n_1,\ldots,n_d}$ be the multiplicity of $S(\mu)_{n_1,\ldots,n_d}$ in $M(W)$. Then the quantity $c_{\mu,n_1+l,\ldots,n_d+l}$ is eventually independent of $l$.\\
\end{proposition}

\begin{proof}
For a partition $\lambda$, once again write $M(\lambda) := M(S^\lambda)$. It suffices to prove the proposition for the module $M(\lambda)$, where $\lambda \vdash m$. In fact, we will show that the proposed stability happens when $l + n_d \geq \lambda_1 + |\lambda| = \lambda_1 + m$.\\

Let $n_d+l \geq \lambda_1 + m$, and assume that $S(\mu)_{n_1+l,\ldots,n_d+l}$ is a constituent of $M(\lambda)_n$. According to Pieri's rule, the multiplicity of $S(\mu)_{n_1+l,\ldots,n_d+l}$ in $M(\lambda)_n$ will be the number of ways to obtain $\mu[n_1+l,\ldots,n_d+l]$ from $\lambda$ in at most $d$ stages, such that blocks are added to distinct columns in each stage. In this proof, we visualize this process as follows. We begin with the tableaux for $\mu[n_1+l,\ldots,n_d+l]$, with a copy of the tableaux of $\lambda$ in the upper left corner darkened. Note that if $\mu[n_1+l,\ldots,n_d+l]$ does not contain $\lambda$, then the multiplicity must be 0. If the relative sizes of $\mu$ and $\lambda$ are such that $\lambda$ is never contained in $\mu[n_1+l,\ldots,n_d+l]$, then our claim is proven. We may therefore assume that this is not the case. Note that, by our assumption on the size of $l$,
\[
\lambda_1 \leq n_d+l - m \leq n_d +l - |\mu|,
\]
from the previous lemma. It follows that a copy of $\lambda$ will appear for our claimed value of $l$.\\

Beginning with our pre-darkened copy of $\lambda$ inside the tableaux of $\mu[n_1+l,\ldots,n_d+l]$, we start darkening boxes in such a way that every column only receives one new darkened box and, when we are finished, the union of all darkened boxes is the tableaux of some partition. We repeat this process precisely $d$-times. Pieri's rule implies that the coefficient of $S(\mu)_{n_1+l,\ldots,n_d+l}$ in $M(\lambda)$ will be the number of ways that this procedure terminates with the entirety of $\mu[n_1+l,\ldots,n_d+l]$ darkened.\\

Because we only have $d$-steps, it follows that we must darken the first row of $\mu[n_1+l,\ldots,n_d+l]$ up to column $n_d+l- |\mu| = \mu[n_1+l,\ldots,n_d+l]_d$ in the first step. In fact, in the $i$-th step, row $i$ must be filled in to at least column $\mu[n_1+l,\ldots,n_d+l]_d$. It follows that the only choices one has during the entire process are the rate in which the frame $(n_1-n_d,\ldots,n_{d-1}-n_{d})$ is colored in, as well as the rate that the boxes below our pre-colored copy of $\lambda$ are filled. The latter choices are entirely dependent on the relative sizes of $\lambda$ and $\mu$, and one does not gain or lose choices as $l$ grows.\\
\end{proof}

This is the major piece in proving Theorem \ref{genrepstab} for free modules.\\

\begin{proposition}\label{relprojcase}
Let $W$ be a finite dimensional $k[\Sn_m]$-module, write $c_{\mu,n_1,\ldots,n_d}$ for the multiplicity of $S(\mu)_{n_1,\ldots,n_d}$ in $M(W)_n$, and let $\phi^i_n$ denote the induced map of the pair of the standard inclusion $[n] \hookrightarrow [n+1]$ with the color $i$. Then for all $n \geq m$ and all $l \gg 0$,
\begin{enumerate}
\item $\cap_i \ker(\phi^i_n) = 0$;
\item the images of the $\phi^i_n$ span $V_{n+1}$ as an $k[\Sn_{n+1}]$-module;
\item $c_{\mu,n_1 + l,\ldots,n_d+l}$ is independent of $l$;\\
\end{enumerate}
\end{proposition}

\begin{proof}
As before, it suffices to prove the claim in the case of $M(\lambda) := M(S^{\lambda})$. The first two properties above are clear from the definition of $M(\lambda)$. In fact, in this case all the maps $\phi^i_n$ are injective. Proposition \ref{coeffstab} implies the third property.\\
\end{proof}

\subsection{The Proof of Theorem \ref{genrepstab}}

We begin with the following Theorem of Sam and Snowden.\\

\begin{theorem}[\cite{SS3}]\label{ktheory}
The Grothendieck group of finitely generated $\FI_d$-modules over $k$ is generated by classes of free objects, as well as classes of modules whose Hilbert function is $o(d^n)$.\\
\end{theorem}

Recall from Theorem \ref{noeth} that if $V$ is a finitely generated $\FI_d$-module, then there exists polynomials $p_1^V,\ldots,p_d^V \in \Q[x]$ such that $\dim_k(V_n) = p_1^V(n) + p_2V(n)2^n + \ldots + p_d^V(n)d^n$. One can interpret the above theorem of Sam and Snowden as saying that the $p_d(n)d^n$ term in this formula is realized in the Grothendieck group by sums of classes of free objects, while all lower terms arise from some slower growing modules.\\

\begin{proof}[Proof of Theorem \ref{genrepstab}]
First assume that $V$ satisfies all conditions of Theorem \ref{genrepstab}, and that $V_n$ is finite dimensional for all $n$. Then there is some $N$ such that the images of the $\phi_n^i$ span $V_{n+1}$ as a $k[\Sn_{n+1}]$-module for all $n \geq N$. For each $i \leq N$, let $\{v_{m,i}\}_{m=1}^{\dim_k(V_i)}$ be a basis for $V_i$. Then $\cup_{i,m} \{v_{m,i}\}$ is a generating set for $V$ by definition. This shows that $V$ is finitely generated.\\

Conversely assume that $V$ is finitely generated. The fact that $V_n$ is finite dimensional for all $n$ is clear from this. Next, set $K_n := \cap_i \ker(\phi_n^i)$. We make $K_n$ into an $k[\Sn_n]$-module by restricting the action of $V_n$, and we make the collection of $K_n$ into an $\FI_d$-module by making all the transition maps trivial. It follows from definition that $K$ is a submodule of $V$. Theorem \ref{noeth} implies that $K$ must have a finite generating set $\{v_i\}$. Because the transition maps of $K$ are all trivial, finite generation implies that it must be supported in only finitely many degrees. In particular, $K_n = \cap_i \ker(\phi_n^i) = 0$ for $n \gg 0$.\\

The definition of finite generation implies the second statement. It therefore remains to show that $c_{\lambda,n_1+l,\ldots,n_d+l}$ is eventually constant. Theorem \ref{ktheory} implies that it suffices to show the claim when $V = M(W)$, for some $k[\Sn_m]$-module $W$, or when the Hilbert function of $V$ is $o(d^n)$. Proposition \ref{relprojcase} implies the statement when $V = M(W)$. We therefore assume that $V$ has a Hilbert function which is $o(d^n)$. We will prove the stronger claim that $c_{\lambda,n_1+l,\ldots,n_d+l} = 0$ for $l \gg 0$. This is clear if $d = 1$, and so we assume that $d > 1$.\\

Our strategy to prove this stronger claim is to show that $\dim_k(S(\lambda)_{n_1+l,\ldots,n_d+l})$ grows, as a function of $l$, faster than our assumptions permit.\\

Let $\lambda$ be a partition, fix $n_1 \geq \ldots \geq n_d \geq |\lambda| + \lambda_1$, and consider the tableaux associated to $\lambda[n_1+l,\ldots,n_d + l]$ for some positive integer $l$. Define $s := n_1+l-|\lambda|$ and $r := \sum_{i} n_1-n_i$. We note that $r$ is the number of blocks missing to make the first $d$ rows of $\lambda[n_1+l,\ldots,n_d + l]$ into a rectangle, while $s$ is the length of the first row. The hook formula tells us that 
\[
F(l) := \dim_k(S(\lambda)_{n_1+l,\ldots,n_d + l}) = \frac{(ds+|\lambda|-r)!}{\prod_{(i,j) \in \lambda[n_1+l,\ldots,n_d+l]}H(i,j)}.
\]
Observe that we may multiply this expression by the polynomial $q(l) := \prod_{i = 1}^{(d-1)|\lambda| + r}(ds + |\lambda| -r +i)$ to obtain
\[
q(l)F(l) = \frac{(ds+d|\lambda|)!}{\prod_{i,j}H(i,j)}.
\]
The function $q(l)$ is indeed a polynomial, as $r$ is be independent of $l$.\\

At any box $(i,j)$ in the first $d$ rows of $\lambda[n_1+l,\ldots,n_d + l]$ the hook with corner $(i,j)$ is shorter horizontally than a hook which stretches to the final column $s$, and it's shorter vertically than if it stretched down past the remaining $d-i$ padded rows, and then through $|\lambda|$ blocks. This gives a lower bound,
\[
q(l)F(l) \geq \frac{(ds+d|\lambda|)!}{(d-1+ s +|\lambda|)(d-1+ (s-1) + |\lambda|)\cdots (d-1 + 1+ |\lambda|)(d-2 + s+|\lambda|)\cdots (1+|\lambda|)C_\lambda}
\]
where $C_\lambda$ is a constant depending only on $\lambda$, which is the product of the sizes of hooks completely contained inside $\lambda$. Multiplying both sides of this inequality by terms of the form $(d-m + s + |\lambda|)$ with $1 \leq m \leq d$, the constant $C_\lambda$, and the constant $\frac{1}{\prod_{m = 0}^{d-1} (d-m + |\lambda|)!}$, this becomes
\[
\widetilde{q}(l)F(l) \geq \frac{(ds+d|\lambda|)!}{(s+|\lambda|)!(s+|\lambda|)!\cdots (s+|\lambda|)!} = \binom{ds+d|\lambda|}{s+|\lambda|,s+|\lambda|,\ldots,s+|\lambda|}
\]
for some polynomial $\widetilde{q}(l)$.\\

Elementary combinatorics tells us that the multinomial coefficient on the right hand side of the above inequality is the largest among all coefficients with power $ds+d|\lambda|$ in $d$ parts. This implies that it is bounded from below by the average of all such coefficients. The sum of all multinomial coefficients of a given power is an exponential, and the total number of such coefficients is a polynomial in the power. Thus,
\[
\widehat{q}(l)F(l) \geq d^{ds+d|\lambda|}
\]
for some polynomial $\widehat{q}(l)$. By assumption, our module has dimension growth $o(d^n)$, and $s$ grows linearly with $l$. This shows that the multiplicity $c_{\lambda,n_1+l,\ldots,n_d+l}$ must eventually be zero, as desired.\\
\end{proof}

\subsection{The Multiplicities of the Trivial Representation}

In this section we consider multiplicities of the trivial representations within in the module $V^{<d}$. In the next section, we will extend the main result of this section (Theorem \ref{trivmult}) to multiplicities of more general irreducible representations (see Theorem \ref{polystab}). We will find that the proof of the more general case amounts to a twist of the trivial representation case, and so we treat this case separately.\\

To better understand the multiplicity of the trivial representation, we make use of the coinvariants functor first introduced in \cite{CEF} for $\FI$-modules. The motivation for most of what follows can be found in \cite[Section 3.1]{CEF}.\\

\begin{definition}
Let $R$ denote the polynomial ring $k[x_1,\ldots,x_d]$, and write $\text{Mod}_R^{gr}$ for the category of non-negatively graded $R$-modules. The \textbf{coinvariants functor} $\Phi:\FI_d\Mod \rightarrow \text{Mod}_R^{gr}$ is defined by 
\[
\Phi(V) = \bigoplus_n V_n \otimes_{k[\Sn_n]} S^{(n)}.
\]
where $S^{(n)} = k$ is the trivial representation.
\end{definition}

One should think that taking coinvariants identifies transition maps whenever their coloring uses the same colors with the same frequency. After doing this, the maps $\phi^i_n$, as in the statement of Theorem \ref{genrepstab}, become the action by the variables $x_i$ in $R$. One should also observe that if $c_n$ is the multiplicity of the trivial representation in $V_n$, then $\dim_k \Phi(V)_n = c_n$.\\

We briefly record the following elementary properties of $\Phi$.\\

\begin{proposition}
The functor $\Phi:\FI_d\Mod \rightarrow \text{Mod}_R^{gr}$ enjoys the following properties.
\begin{enumerate}
\item $\Phi$ is exact;
\item $\Phi(M(m)) = R(-m)$, where $R(-m)$ denotes the graded twist of $R$ by $-m$;
\item If $V$ is finitely generated, then so is $\Phi(V)$.\\
\end{enumerate}
\end{proposition}

\begin{proof}
The first statement follows immediately from the fact that $k$ is a field of characteristic 0.\\

Recall that for a composition $a = (a_1,\ldots,a_d)$ of $n-m$ of length $d$, we write $M(m)_{n,a}$ to denote the submodule of $M(m)_n$ spanned by basis vectors $(f,g)$, such that $|g^{-1}(i)| = a_i$. Each $M(m)_{n,a}$ is a permutation representation of $\Sn_n$, where $\Sn_n$ acts transitively. Therefore, $M(m)_n \otimes k$ has a basis in bijection with the set of all such compositions. The map $\phi_n^i$ acts on such a basis vector $(a_1,\ldots,a_n)$ via $\phi_n^i(a_1,\ldots,a_n) = (a_1,\ldots,a_{i-1},a_i+1,a_{i+1},\ldots,a_n)$. There is therefore a natural isomorphism of graded $R$ modules between $\Phi(M(m))$ and $R(-m)$ via
\[
(a_1,\ldots,a_d) \mapsto x_1^{a_1}\cdots x_n^{a_n}.
\]

The final statement follows immediately from the first two, and the definition of finite generation.\\
\end{proof}

It is a classically known fact that if $M$ is a finitely generated graded $R$-module, then there is a polynomial $p \in \Q[x]$ such that $\dim_k M_n = p(n)$ for all $n \gg 0$ (See \cite{E} for a proof). The discussion of this section therefore implies the following.\\

\begin{theorem}\label{trivmult}
Let $V$ be a finitely generated $\FI_d$-module, and write $c_n$ for the multiplicity of the trivial representation in $V_n$. Then there is a polynomial $p \in \Q[x]$ of degree $\leq d-1$ such that for all $n \gg 0$, $c_n = p(n)$.\\
\end{theorem}

\subsection{The proof of Theorem \ref{polystab}}\label{thmb}

To finish the paper, we must prove Theorem \ref{polystab}. This theorem was proven in the case wherein $\lambda = (n)$ in the previous section, and we will find that the proof of the general case follows from this. The main trick for the reduction of the general case to the previous case was shown to the author by Steven Sam. The author would like to send his thanks to Professor Sam for aiding him in this way.\\

We begin with some notation from the representation theory of categories.\\

\begin{definition}
Let $\Ca$ and $\Ca'$ be (small) categories. A \textbf{$\Ca$-module} over a ring $R$ is a functor $V:\Ca \rightarrow R\Mod$. We say that a $\Ca$-module is \textbf{finitely generated} if there is a finite collection of elements $\{v_i\} \subseteq \bigsqcup_{X \in \Ca} V(\Ca)$ which is contained in no proper submodule. Assume that $V$ is a $\Ca$-module over some ring $R$, while $V'$ is a $\Ca'$-module over the same ring. Then the \textbf{(exterior) tensor product of $V$ an $V'$}, denoted $V \boxtimes V'$, is the $\Ca \times \Ca'$-module defined on points by
\[
(V \boxtimes V')(X,X') = V(X) \otimes_R V'(X')
\]
\text{}\\
\end{definition}

For the remainder of the section, We let $V$ be a finitely generated $\FI_d$-module over a field $k$ of characteristic 0. We also fix a partition $\lambda \vdash m$, and set $c_{\lambda,n}$ to be the multiplicity of $S(\lambda)_n$ in $V_n$. The primary strategy for proving Theorem \ref{polystab} is to twist a given $\FI_d$-module with an appropriately chosen $\FI$-module, so that the multiplicity of the trivial representation in the resulting module is equal to $c_{\lambda,n}$ at this point the previous section's results will imply Theorem \ref{polystab}.\\

In \cite[Proposition 3.4.1]{CEF}, Church, Ellenberg and Farb define a finitely generated $\FI$-module $S(\lambda)$ with the property that
\[
S(\lambda)_n = \begin{cases} 0 &\text{ if $n < m$}\\ S^{\lambda[n]} &\text{ otherwise.}\end{cases}
\]
This is the $\FI$-module with which we will twist $V$. To accomplish this formally, we first need some technical lemmas from the Gr\"obner theory of Sam and Snowden \cite{SS}.\\

\begin{definition}
Let $\Ca$ and $\Ca'$ be two categories, and let $\Psi:\Ca \rightarrow \Ca'$ be a (covariant) functor between these categories. Then we say that $\Psi$ \textbf{satisfies property (F)} if given any object $x$ of $\Ca'$, there is a finite list of objects $y_1,\ldots,y_r$ of $\Ca$ and morphisms $f_i:x \rightarrow \Psi(y_i)$, such that for any object $y$ of $\Ca$ and any morphism $f:x \rightarrow \Psi(y)$ in $\Ca'$, there is a morphism $g:y_i \rightarrow y$ in $\Ca$ for some $i$ such that $f = \Psi(g) \circ f_i$.\\
\end{definition}

Property $(\textbf{F})$ was initially considered by Sam and Snowden in \cite{SS}. In this work they show that this property is the key tool for connecting stability phenomena in the study of $\Ca'$-modules to stability phenomena in the study of $\Ca$-modules. For our purposes, we will only need the following lemmas.\\

\begin{lemma}[\cite{SS}, Proposition 3.2.3]\label{lem1}
Let $\Psi:\Ca \rightarrow \Ca'$ be a functor, and write $\Psi^{\as}$ for the natural pullback $\Psi^{\as}:\Ca'\Mod \rightarrow \Ca\Mod$. Then $\Psi$ satisfies property (\textbf{F}) if and only if $\Psi^{\as}$ maps finitely generated $\Ca'$-modules to finitely generated $\Ca$-modules.\\
\end{lemma}

\begin{lemma}\label{lem2}
Let $\Psi:\FI_d \rightarrow \FI_d \times \FI$ denote the functor defined by the assignments
\[
\Psi([n]) = ([n],[n]), \hspace{1cm} \Psi((f,g)) = ((f,g),f).
\]
Then $\Psi$ satisfies property (\textbf{F}).\\
\end{lemma}

\begin{proof}
Let $([r],[s])$ be an object in $\FI_d \times \FI$, and assume that $r < s$. Looking at the definition of property (\textbf{F}), we set $y_i = [s]$ and allow the $f_i:([r],[s]) \rightarrow \Psi(y_i)$ to vary across all of the (finitely many) morphisms $([r],[s]) \rightarrow ([s],[s]) = \Psi(y_i)$ in $\FI_d \times \FI$. Now given any $[l]$, and any morphism $f:([r],[s]) \rightarrow \Psi([l]) \in \FI_d \times \FI$, we set $(h,h'):[s] \rightarrow [l] \in \FI_d$ to be the pair of the standard inclusion with a coloring chosen to agree with the coloring associated to $f$ on the compliment of $[s]$ in $[l]$. It follows there is some map in $\FI_d \times \FI$, $g:([r],[s]) \rightarrow ([s],[s])$ such that $\Psi(h,h') \circ g = f$. Such a $g$ will be included among the $f_i$ by how the $f_i$ were chosen. The case where $r>s$ is the same.\\
\end{proof}

This is all we need to prove Theorem \ref{polystab}.\\

\begin{proof}[proof of Theorem \ref{polystab}]
Consider the $(\FI_d \times \FI)$-module $V \boxtimes S(\lambda)$. This module is finitely generated, as both $V$ and $S(\lambda)$ are. Lemmas \ref{lem1} and \ref{lem2} imply that $\Psi^{\as}(V \boxtimes S(\lambda))$ is also a finitely generated $\FI_d$-module. By construction, the multiplicity of the trivial representation in $\Psi^{\as}(V \boxtimes S(\lambda))_n$ is precisely $c_{\lambda,n}$. Theorem \ref{trivmult} now implies Theorem \ref{polystab}.\\
\end{proof}

\end{document}